\documentclass[reqno,10pt]{amsart}
\usepackage{amsfonts, amsmath, amssymb, amsthm, color}
\allowdisplaybreaks[1]
\newtheorem{thm}{Theorem}[section]
\newtheorem{lemma}[thm]{Lemma}
\newtheorem{prop}[thm]{Proposition}

\newtheorem{rmk}[thm]{Remark}
\theoremstyle{definition}
\numberwithin{equation}{section}
\newcommand{\rr}{\mathbb R}
\newcommand{\al}{\alpha}
\newcommand{\de}{\delta}
 \newcommand{\eps}{\varepsilon}
\newcommand{\la}{\lambda}
\newcommand{\Pda}{\mathcal P(d\alpha)}

\newcommand{\calP}{\mathcal P}

\newcommand{\ee}{\varepsilon}

\newcommand{\Om}{\Omega}

\newcommand{\C}{{\mathbb C}}
\newcommand{\N}{{\mathbb N}}
\newcommand{\R}{{\mathbb R}}

\newcommand{\un}{u_n}
\newcommand{\wn}{w_n}

\newcommand{\lan}{\lambda_n}
\newcommand{\xn}{x_n}
\newcommand{\an}{\alpha_n}
\newcommand{\In}{\mathcal I_n}
\newcommand{\Vn}{V_n}
\newcommand{\sn}{\sigma_n}
\newcommand{\rn}{\rho_n}
\newcommand{\tOn}{\widetilde\Omega_n}
\newcommand{\twn}{\widetilde w_n}
\newcommand{\tw}{\widetilde w}
\newcommand{\tVn}{\widetilde V_n}
\newcommand{\trn}{\widetilde\rho_n}

\newcommand{\I}{[0,1]}

\newcommand{\tP}{\Phi}
\renewcommand{\ln}{\log}
\newcommand{\abs}[1]{\left\vert#1\right\vert}
\DeclareMathOperator{\diam}{diam}
\DeclareMathOperator{\dist}{dist}
\DeclareMathOperator{\supp}{supp}
\DeclareMathOperator{\meas}{meas}
\def\sideremark#1{\ifvmode\leavevmode\fi\vadjust{\vbox to0pt{\vss
 \hbox to 0pt{\hskip\hsize\hskip1em
 \vbox{\hsize3cm\tiny\raggedright\pretolerance10000
  \noindent #1\hfill}\hss}\vbox to8pt{\vfil}\vss}}}%

\begin{document}
\title[Existence of stationary turbulent flows]{Existence of stationary turbulent flows with variable 
positive vortex intensity}
\author{F.~De Marchis${}^\ast$}\thanks{${}^\ast$Corresponding author}
\address[F.~De Marchis] {Dipartimento di Matematica, Universit\`{a} di Roma Sapienza, P.le Aldo Moro 5, 00185 Roma, Italy}
\email{demarchis@mat.uniroma1.it}
\author{T.~Ricciardi}
\address[T.~Ricciardi] {Dipartimento di Matematica e Applicazioni, 
Universit\`{a} di Napoli Federico II, Via Cintia, Monte S.~Angelo, 80126 Napoli, Italy}
\email{tonricci@unina.it}
\begin{abstract}
We prove the existence of stationary turbulent flows with arbitrary positive vortex circulation 
on non simply connected domains.
Our construction yields solutions for all real values of the inverse temperature 
with the exception of a quantized set,
for which blow-up phenomena may occur. 
Our results complete the analysis initiated in \cite{RiZeJDE}.
\end{abstract}
\date{}
\maketitle
\section{Introduction and main results}
Motivated by the statistical mechanics description of turbulent 2D Euler flows in equilibrium, 
we are interested in the existence of solutions to the following problem:
\begin{equation}
\label{eq:Neribis}
\tag*{$(*)_\lambda$}
\left\{\begin{aligned}
-\Delta u =&\lambda\frac{\int_{[0,1]}\al e^{\al u}\,\mathcal P(d\al)}{\iint_{[0,1]\times\Omega}e^{\al u}\,\mathcal P(d\al) dx}&&\hbox{in}\ \Omega\\
   u =&0&&\hbox{on}\ \partial\Omega,
\end{aligned}\right.
\end{equation}
where $\Omega\subset\mathbb R^2$ is a smooth bounded domain, $\la>0$ is a constant
and $\calP\in\mathcal M([0,1])$ is a Borel probability measure.
Problem~\ref{eq:Neribis} was derived by Neri~\cite{Neri} within Onsager's pioneering framework \cite{Onsager},
with the aim of including the case of variable vortex intensities.
More precisely, in \cite{Neri}
the following mean field equation is derived:
\begin{equation}
\label{eq:Neri}
\left\{\begin{aligned}
-\Delta v =&\frac{\int_{[-1,1]}r e^{-\beta r v}\,\mathcal P(dr)}{\iint_{[-1,1]\times\Omega}e^{-\beta rv}\,\mathcal P(dr) dx}&&\hbox{in}\ \Omega\\
v =&0&&\hbox{on}\ \partial\Omega.
\end{aligned}\right.
\end{equation}
Here,  $v$ is the mean field stream function of an incompressible 
turbulent Euler flow,
the Borel probability measure $\calP\in\mathcal M([-1,1])$ 
describes the vortex intensity distribution 
and $\beta\in\rr$ is a constant related to the inverse temperature. 
The mean field equation \eqref{eq:Neri} is derived from the classical Kirchhoff-Routh Hamiltonian 
for the $N$-point vortex
system:
\begin{equation*}
H^N(r_1,\ldots,r_N, x_1,\ldots,x_N)
=\sum_{i\neq j}r_ir_jG(x_i,x_j)+\sum_{i=1}^Nr_i^2H(x_i,x_i),
\end{equation*}
in the limit $N\to\infty$,
under the \emph{stochastic} assumption that the $r_i$'s are independent identically distributed
random variables with distribution $\mathcal P$.
In the above formula, for $x,y\in\Omega$, $x\neq y$,
$G(x,y)$ denotes the Green's function defined by
\begin{equation*}
\left\{
\begin{aligned}
&-\Delta G(\cdot,y)=\delta_y&&\hbox{in\ }\Omega\\
&G(\cdot,y)=0&&\hbox{on\ }\partial\Omega
\end{aligned}
\right.
\end{equation*}
and $H(x,y)$ denotes the regular part of $G$, i.e.
\begin{equation}
\label{def:H}
H(x,y)=G(x,y)+\frac{1}{2\pi}\log|x-y|.
\end{equation}
Setting $u:=-\beta v$ and $\lambda=-\beta$, and assuming that 
\begin{equation}
\label{assumpt:positivesupp}
\supp\calP\subset[0,1], 
\end{equation}
problem~\eqref{eq:Neri}
takes the form~\ref{eq:Neribis}.
We recall that
\begin{equation*}
\supp\calP:=\{\al\in[-1,1]:\ \calP(N)>0\ \mbox{for any open neighborhood $N$ of $\al$}\}.
\end{equation*}
Assumption~\eqref{assumpt:positivesupp} corresponds to the 
case of physical interest where
all vorticities have the same orientation.
\par
We observe that without loss of generality we may assume
\begin{equation}
\label{assumpt:suppP}
1\in\supp\mathcal P.
\end{equation}
Indeed, suppose that $\sup\supp\calP=\bar\al\in(0,1)$.
Then, \ref{eq:Neribis} is equivalent to
\begin{equation*}
\left\{\begin{aligned}
-\Delta u =&\lambda\frac{\int_{[0,\bar\al]}\al e^{\al u}\,\mathcal P(d\al)}
{\iint_{[0,\bar\al]\times\Omega}e^{\al u}\,\mathcal P(d\al) dx}&&\hbox{in}\ \Omega\\
u =&0&&\hbox{on}\ \partial\Omega.
\end{aligned}\right.
\end{equation*}
By the change of variables $\al=\al'\bar\al$, $\bar\calP(A)=\calP(\bar\al A)$
for all Borel sets $A\subset[0,1]$, and setting $\bar u=\bar\al u$, we find that
$\bar u$ satisfies
\begin{equation*}
\left\{
\begin{aligned}
-\Delta\bar u =&\bar\al^2\lambda
\frac{\int_{[0,1]}\al' e^{\al'\bar u}\,\bar\calP(d\al')}
{\iint_{[0,1]\times\Omega}e^{\al'\bar u}\,\bar\calP(d\al')dx}&&\hbox{in}\ \Omega\\
u =&0&&\hbox{on}\ \partial\Omega,
\end{aligned}
\right.
\end{equation*}
which is nothing but $(*)_{\bar\alpha^2\lambda}$, with $\bar\calP$ satisfying \eqref{assumpt:suppP}.
Henceforth, we always assume \eqref{assumpt:suppP}.
\par
When $\mathcal P(d\al)=\delta_1(d\al)$ problem~\ref{eq:Neribis}
reduces to the \emph{standard} mean field problem
\begin{equation*}
\left\{\begin{aligned}
-\Delta u =&\lambda\frac{e^{u}}{\int_{\Omega}e^{u}\,dx}&&\hbox{in}\ \Omega\\
   u =&0&&\hbox{on}\ \partial\Omega,
\end{aligned}\right.
\end{equation*}
which has been extensively analyzed, see e.g. \cite{Lin} and the references therein. 
In the context of turbulence, the case $\calP(d\al)=\de_{1}(d\al)$ was developed in \cite{CLMP}, see also \cite{BarDem2}.
\par
Problem~\ref{eq:Neribis} admits a variational formulation.
Indeed, solutions to \ref{eq:Neribis} correspond to critical points in $H_0^1(\Omega)$
for  the functional
\begin{equation}
\label{eq:Nerifunctional}
J_\lambda(u)=\frac{1}{2}\int_\Omega|\nabla u|^2\,dx-\lambda\ln\Big(\iint_{[0,1]\times\Omega}e^{\al u}\,\Pda dx\Big).
\end{equation}
Whether or not the optimal value of $\la$ such that $J_\la$ is bounded from below 
depends on $\calP$ was 
raised as an open question in \cite{Suzukibook}, p.~192, in relation to 
other apparently similar models for which such a dependence holds true.
However, it was noticed in \cite{RZDIE} that, in fact, $J_\la$ may be viewed as a perturbation
of the standard Moser-Trudinger functional \cite{Moe, Tru} and that, under assumption~\eqref{assumpt:suppP},
such an optimal value of $\la$ is exactly $8\pi$ independently of $\calP$. 
More precisely, it was already observed in \cite{Neri} that $J_\la$ is bounded from below on $H_0^1(\Omega)$ if
$\la\le 8\pi$.
Consequently, the existence of minimizing solutions for \ref{eq:Neribis} was obtained in \cite{Neri} in the 
subcritical range
$\la\in(0,8\pi)$.
In \cite{RiZeJDE} the existence of solutions to \ref{eq:Neribis} was obtained in the supercritical range $\la\in(8\pi,16\pi)$ 
under the \emph{non-degeneracy} assumption 
\begin{equation}
\label{assumpt:nondeg}
\calP(\{1\})>0.
\end{equation}
If \eqref{assumpt:nondeg} is satisfied, problem~\ref{eq:Neribis} may be written in the form $-\Delta u=\rho f(u)$,
with $f(t)=e^t+o(e^t)$ as $t\to+\infty$, $\rho>0$,
and thus it fits into the framework considered in \cite{NaSu, Ye}.
In particular, if \eqref{assumpt:nondeg} is satisfied, 
the techniques in \cite{NaSu, Ye} may be applied
to obtain the mass quantization of concentrating solution sequences. 
On the other hand, the case $\calP(\{1\})=0$ requires extra care.
\par
Thus, our aim in this note is to complete the existence result in \cite{RiZeJDE}
by establishing the existence of solutions to \ref{eq:Neribis} in the supercritical regime, 
\textit{without assuming} \eqref{assumpt:nondeg} and for \textit{all} values of $\lambda$ for which compactness of solution sequences holds. 
\par
In order to state our results precisely, 
we recall that by the Brezis-Merle concentration compactness theory \cite{BM}, 
as adapted in \cite{RZDIE}, an $L^\infty$-unbounded sequence $u_n$ of solutions to $(*)_{\lambda_n}$ necessarily concentrates at a finite number of 
points in $\Omega$, namely
\begin{equation}
\label{eq:BMintro}
\la_n\frac{\int_{[0,1]}\al e^{\al u_n}\,\calP(d\al)}{\iint_{[0,1]\times\Omega}e^{\al u_n}\,\calP(d\al)dx}
\stackrel{\ast}{\rightharpoonup}\sum_{i=1}^m n_i\de_{p_i}(dx)+s(x)\,dx,
\end{equation}
weakly in the sense of $\mathcal M(\Omega)$,
for some $m\in\mathbb N$, $p_i\in\Omega$, $n_i\ge4\pi$, $i=1,2,\ldots,m$, and $s\in L^1(\Omega)$.
Our first aim is to improve \eqref{eq:BMintro} by showing that, actually, there holds $n_i=8\pi$ for all $i=1,2,\ldots,m$, 
moreover $s\equiv0$ and $\la_0\in8\pi\mathbb N$. Namely, we establish the following \emph{mass quantization} result.
\begin{thm}
\label{thm:mq}
Assume that $\calP$ satisfies \eqref{assumpt:suppP}.
Let $\la_n\to\la_0$ and let $\un$ be a concentrating sequence of solutions to $(*)_{\lambda_n}$. Then, there exist $p_i\in\Omega$, $i=1,2,\ldots,m$,
such that, up to subsequences, 
\begin{equation}
\label{eq:Thm1}
\la_n\frac{e^{\al u_n}\,\calP(d\al)}{\iint_{[0,1]\times\Omega}e^{\al u_n}\,\calP(d\al)dx}
\stackrel{\ast}{\rightharpoonup}8\pi\sum_{i=1}^m\de_{1}(d\al)\de_{p_i}(dx),
\end{equation}
weakly in the sense of $\mathcal M([0,1]\times\Omega)$.
In particular, $\la_0\in8\pi\mathbb N$.
\end{thm}
In the \emph{non-degenerate} case \eqref{assumpt:nondeg}, Theorem~\ref{thm:mq} was established in
\cite{RiZeJDE}, see also \cite{RiZeANS} for an alternative proof. 
\par
Via Theorem~\ref{thm:mq} and a min-max construction, we shall then obtain the existence result for solutions to \ref{eq:Neribis}.
For the existence result we need to assume that $\Om$ is topologically non-trivial, namely that:
\begin{equation}
\label{assumpt:Omega}
\mbox{$\Omega$ is non-simply connected.}
\end{equation}
Our existence result is the following.
\begin{thm}
\label{thm:main}
Assume that $\calP$ satisfies \eqref{assumpt:suppP}.
Assume that $\Om$ satisfies \eqref{assumpt:Omega}.
Then, for every $\la\in\cup_{k\in\mathbb N}(8\pi k,8\pi(k+1))$ 
there exists a solution to problem~\ref{eq:Neribis}.
\end{thm}
We shall obtain the solutions as saddle-type critical points for the Euler-Lagrange functional
$J_\la$ defined in \eqref{eq:Nerifunctional}, following the variational scheme introduced in \cite{BarDem}, 
see also \cite{DJLW}.
It will be clear from the proof that, alternatively, 
we could follow the variational approach introduced in \cite{DjadliMalchiodi}, see also \cite{Djadli}. 
\par
The article is organized as follows.
In Section~\ref{sec:prelims} we recall some known results and we establish some
necessary lemmas.
In Section~\ref{sec:blow-up} we obtain some blow-up results and we 
prove Theorem~\ref{thm:mq}.
In Section~\ref{sec:var} we set up the variational construction and we prove Theorem~\ref{thm:main}.
In the Appendix we show that a suitable rescaling yields a Liouville bubble profile in the limit.
This fact, although not needed in the variational construction, provides an intuitive justification to the quantization
of the values of $\la$ for which blow-up may occur.
In the \lq\lq degenerate" case $\calP(\{1\})=0$, the appropriate rescaling parameters depend on $\calP$
in a non-trivial way.
\subsection*{Notation}
We denote by $C>0$ a general large constant whose actual value is allowed to vary.
We denote by $\mathbb N$ the set of positive integers.
When the integration variable is clear from the context we omit it.
Henceforth, we denote $I:=[0,1]$.
\section{Preliminary results}
\label{sec:prelims}
For the sake of completeness, we collect in this section some preliminary results 
of various nature which will be used in the sequel.
\subsection{Concentration-compactness principle}
\label{subsection:BM}
We recall the Brezis-Merle blow-up theory \cite{BM},
as adapted to \ref{eq:Neribis} in \cite{RZDIE}.
Let us define the sequence of measures $\nu_n\in\mathcal M(\Omega)$ by
\[
\nu_n(dx):=\la_n\frac{\int_{\I}\al e^{\al u_n(x)}\,\calP(d\al)}{\iint_{[0,1]\times\Omega}e^{\al u_n(x)}\,\calP(d\al)dx}\,dx.
\]
Then, the following alternative holds true.
\begin{lemma}[Brezis-Merle alternative]
\label{lem:BM}
Let $\un$ be a sequence of solutions to $(*)_{\lambda_n}$
with $\la_n\to\la_0$.
Then, up to subsequences, exactly one of the following alternatives holds:
\begin{enumerate}
\item [(i)]
(Compactness) There exists a solution $u_0\in H_0^1(\Omega)$ to $(*)_{\lambda_0}$ such that $u_n\to u_0$ in any relevant norm;
\item[(ii)]
(Concentration) There exists a finite, non-empty blow-up set $\mathcal S=\{p_1,\ldots,p_m\}\subset\Omega$
such that $\un\in L^\infty_{\mathrm{loc}}(\Omega\setminus\mathcal S)$ and
\begin{equation}
\label{eq:nuconv}
\nu_n\stackrel{\ast}{\rightharpoonup}\sum_{i=1}^m n_i\delta_{p_i}(dx)+s(x)\,dx
\end{equation}
for some $n_i\ge4\pi$, $i=1,\ldots,m$ and for some $s\in L^1(\Omega)$.
\end{enumerate}
\end{lemma}
\begin{proof}
We first observe that, in view of 
the two-dimensional argument in \cite{GNN} p.~223,
there exists $\eps>0$, depending only on $\Omega$, such that $\un$ has no
stationary point in an $\eps$-neighborhood of $\partial\Om$.
Consequently,
blow-up does not occur on the boundary $\partial\Omega$.
\par
We adapt Theorem~3, p.~1237 in \cite{BM} to our case.
Let
\[
W_n(x):=\la_n\frac{\int_{[0,1]}\al e^{-(1-\al)u_n}\,\Pda}{\iint_{[0,1]\times\Omega}e^{\al' u_n}\,P(d\al')}.
\]
Then, problem $(*)_{\la_n}$ takes the form
\begin{equation*}
\left\{
\begin{aligned}
-\Delta u_n=&W_n(x)e^{u_n}&&\mathrm{in\ }\Omega\\
u_n=&0&&\mathrm{on\ }\partial\Omega.
\end{aligned}
\right.
\end{equation*}
By the maximum principle, we have $u_n\ge0$ and hence,
\[
0\le W_n(x)\le\frac{\la_n}{|\Omega|}\int_{[0,1]}\al\,\Pda
\qquad\mbox{for all }x\in\Omega.
\]
Moreover,
\begin{equation}
\label{est:Weu}
\int_\Omega W_n(x)e^{u_n}\,dx\le\la_n.
\end{equation}
Therefore, assumptions~(21)--(22) in
\cite{BM}, Theorem~3, are satisfied with the exponent $p=+\infty$.
Consequently, it is readily seen that the proof of Theorem~3 in \cite{BM}
may be adapted in order to prove that either alternative~(i) holds true,
or there exists a finite set $\mathcal S\subset\Omega$ such that, up to subsequences,
$\un$ is bounded in $L^\infty_{\mathrm{loc}}(\Omega\setminus\mathcal S)$.
In the latter case it follows that \eqref{eq:nuconv} holds true,
i.e., alternative (ii) is satisfied.
\end{proof}
\begin{rmk}
\label{rem:s=0}
In the statement of Theorem~3 in \cite{BM},
a third assumption (23) is made on the sequence $\un$ of solutions to $(*)_{\lambda_n}$, 
namely it is assumed that $\sup_n\int_\Omega e^{\un}<+\infty$.  
Since in our case we only have the weaker 
assumption \eqref{est:Weu}, 
we cannot in general directly apply the arguments in \cite{BM} to show
that $s=0$ in \eqref{eq:nuconv}. 
However, if we assume $\calP(\{1\})>0$,
the proof in \cite{BM} may be adapted. 
Indeed, $(*)_{\lambda_n}$ and  \eqref{eq:nuconv} imply that
$\un\to u_0$ weakly in $W_0^{1,q}(\Omega)$, strongly in $L^q(\Omega)$ for any $1\le q<2$,
and a.e.,
where 
\begin{equation}
\label{stimau0}
u_0(x)\ge\sum_{i=1}^m\frac{n_i}{2\pi}(\ln\frac{1}{|x-p_i|}+H(x,p_i)),
\end{equation}
and where $H$ is defined in \eqref{def:H}.
If $\calP(\{1\})>0$ we may estimate
\begin{equation*}
\iint_{[0,1]\times\Omega}e^{\al u_n(x)}\,\calP(d\al)dx\ge
\calP(\{1\})\int_{\Omega}e^{u_n}\,dx.
\end{equation*}
By Fatou's lemma, \eqref{stimau0} and recalling that $n_i\geq4\pi$,
\[
\liminf_{n\to\infty}\int_\Omega e^{\un}\ge\int_\Omega e^{u_0}=+\infty
\]
and consequently
\[
\iint_{[0,1]\times\Omega}e^{\al u_n(x)}\,\calP(d\al)dx\to+\infty.
\]
This implies $s\equiv0$.
\end{rmk}
\subsection{Improved Moser-Trudinger inequality}
\label{subsect-mt}
We shall need an improved Moser-Trudinger inequality for the functional
\eqref{eq:Nerifunctional} defined on the bounded domain $\Om\subset\R^2$.
\par
We recall that
the classical Moser-Trudinger sharp inequality  \cite{Moe,Tru} states that
\begin{equation}
\label{fontana}
C_{MT}:=\sup\left\{\int _\Om e^{4\pi u^2}:\ u\in H^1_0(\Om),\ \|\nabla u \|_2\le 1\right\}<+\infty,
\end{equation}
where the constant $4\pi$ is best possible. 
Moreover, the embedding $u\in H_0^1(\Om)\to e^u\in L^1(\Om)$ is compact.
For a proof, see, e.g., Theorem~2.46 p.~63 in \cite{Aub}.
\par
In view of the elementary inequality
\begin{equation*}
|u|\le\frac{\|\nabla u \|_2^2}{16 \pi}+4\pi  \frac{u^2}{\|\nabla u\|_2^2},
\end{equation*}
we deduce from \eqref{fontana} that
\begin{equation}
\label{ineq:MT}
\log\left(\int_{\Om}  e^{|u|}\,dx\right)\le \frac1{16\pi}\|\nabla u\|_2^2 + \log(C_{MT}),\qquad \mbox{for any \ $u\in H^{1}_0(\Om)$}.
\end{equation}
In particular, the functional
\begin{equation*}
I_\lambda(u)=\frac{1}{2}\|\nabla u\|_2^2-\lambda\log\int_{\Om}e^{u}\,dx
\end{equation*}
is bounded from below for all $\lambda\leqslant 8\pi$, while it is not difficult to 
check that
\begin{equation}
\label{MT-ifunctunbdd}
\inf_{u\in H^1_0(\Om)} I_\lambda (u)=-\infty
\end{equation}
whenever $\lambda >8\pi$.
Indeed, evaluating the functional $I_\lambda$ on the following adaptation of the 
Liouville bubbles defined in \eqref{def:bubble}
below:
\begin{equation*}
u_\eps(x)=
\begin{cases}
\log\frac{(\eps^2+r_0^2)^2}{(\eps^2+|x-x_0|^2)^2},&\hbox{in\ }B_{r_0}(x_0)\\
0, &\hbox{in\ }\Omega\setminus B_{r_0}(x_0),
\end{cases}
\end{equation*}
yields
\begin{equation*}
\int_\Omega|\nabla u_\eps|^2\,dx=16\pi\log\frac{1}{\eps^2}+O(1),\qquad
\log\int_\Omega e^{u_\eps}\,dx=\log\frac{1}{\eps^2}+O(1)
\end{equation*}
so that
\[
I_\la(u_\eps)=(8\pi-\la)\log\frac{1}{\eps^2}+O(1)\to-\infty\qquad\mbox{as $\eps\to0$.}
\]
\par
The arguments above imply that if \eqref{assumpt:suppP} is satisfied, then
analogous results hold for the Neri's functional $J_\lambda$. 
More precisely, we have
\begin{lemma}[Moser-Trudinger inequality]
\label{lem:TM}
Assume that $\calP$ satisfies \eqref{assumpt:suppP}.
Then
\begin{equation}
\label{ineq:JMT}
\log\iint_{I\times\Omega}e^{\al u}\,\Pda dx\le\frac{1}{16\pi}\|\nabla u\|_2^2+\log C_{MT}\qquad\mbox{for any $u\in H_0^1(\Om)$},
\end{equation}
and the functional $J_\lambda(u)$ is bounded from below on $H^1_0(\Om),$ if and only if $\lambda\le8\pi$.
\end{lemma}
Lemma~\ref{lem:TM} was established in \cite{RiZeJDE} for functions $u\in H^1(M)$
satisfying $\int_Mu=0$, where $M$ is a two-dimensional compact Riemannian manifold.
The proof for $u\in H^1_0(\Om)$ is similar. For the sake of completeness, we outline it below. 
\par
On the other hand, in the next Lemma we show that the constant $\frac{1}{16\pi}$ in \eqref{ineq:JMT} may be lowered if the quantity
\[
\frac{\int_{I} e^{\alpha u(x)}\,\calP(d\alpha)}{\iint_{I\times\Om} e^{\alpha u}\,\calP(d\alpha)dx},
\]
which may be interpreted as the \emph{mass} of $u$, is suitably distributed. 
Namely, following ideas of \cite{Aub, cl}, we prove:
\begin{lemma}[Improved Moser-Trudinger inequality]
\label{mtmigliorata} 
Assume that $\calP$ satisfies \eqref{assumpt:suppP}.
Let   $d_0>0$,  $a_0\in (0,1/2)$ and for a fixed positive integer $\ell$, 
let $\Omega_1,\ldots,\Omega_{\ell+1}$ be subsets of 
$\Omega$ satisfying $\dist(\Omega_i,\Omega_j)\geq d_0$, for all $i\neq j$. 
Then, for any $\ee>0$ there exists a constant 
$K=K(\ee,d_0,a_0,\ell)>0$ such that if $u\in H^1_0(\Omega)$ satisfies 
\begin{equation}
\label{omegai}
\frac{\iint_{I\times\Om_i } e^{\alpha u}\calP (d\alpha)dx}{\iint_{I\times\Om} e^{\alpha u}\calP (d\alpha)dx}\ge a_0,
\qquad i=1,\ldots,\ell+1,
\end{equation}
then it holds
\begin{equation}
\label{tesimt}
\log\left(\iint_{I\times\Om} e^{\alpha u}\calP (d\alpha)dx\right)\le\frac{1}{16(\ell+1)\pi-\ee}\int_\Omega|\nabla u|^2\,dx+K.
\end{equation}
\end{lemma}
We begin by outlining the proof of Lemma~\ref{lem:TM}.
\begin{proof}[Proof of Lemma~\ref{lem:TM}]
The ``if" part is immediate and was already used in \cite{Neri} in order
to obtain solutions to \ref{eq:Neribis} for all $\la\in(0,8\pi)$.
Indeed, we have
\begin{equation*}
\iint_{I\times\Om}  e^{\alpha u}\calP(d\alpha)\,dx\le\int_\Om e^{|u|}\,dx
\le C_{MT}\, e^{\frac1{16\pi}\|\nabla u\|_2^2},
\end{equation*}
for all  $u \in H^{1}_0(\Om)$. Therefore $J_\lambda$ is bounded below
if $\lambda\le8\pi$.
On the other hand
the value $8\pi$ is also optimal,
provided that $1\in\supp\calP$.
In order to show it one needs only to prove that
\begin{equation}
\label{opt}
\inf_{u\in H^1_0(\Om),}J_\lambda(u)=-\infty, \qquad \quad\mbox{for any \ $\lambda >8\pi$.}
\end{equation}
Assume \eqref{assumpt:suppP}. Since the functional $I_\lambda(u)$
is unbounded from below for $\lambda >8\pi$, then also the functional
\begin{equation*}
I_\lambda(u)_{|_{u\ge0}}=\frac{1}{2}\|\nabla u\|_2^2-\lambda\log\int_{\Om}e^{u}\,dx, \qquad u\in H^1_0(\Om),\,\, u\ge0
\end{equation*}
is unbounded below for $\lambda >8\pi$. At this point we observe that for every $0<\delta<1$ and $u\ge0$, $u\in H^1_0(\Om),$ we have:
\begin{equation*}
\begin{split}
J_\lambda (u) &= \frac 12 \|\nabla u\|_{2}^2-\lambda \log \iint_{I\times\Om} e^{\alpha u}\calP (d\alpha )dx
\le\frac 12 \|\nabla u\|_{2}^2-\lambda \log \iint_{[1-\delta,1]\times \Om}  e^{\alpha u}\calP(d\alpha)  dx\\
&\le\frac 12\|\nabla u\|_{2}^2-\lambda \log \int_\Om  e^{(1-\delta) u} dx -\lambda \log (\calP([1-\delta,1]))\\
&= \frac{1}{(1-\delta)^2} \left[\frac 12 \|(1-\delta)\nabla  u\|_2^2-\lambda 
(1-\delta)^2 \log\left(\int_\Om  e^{(1-\delta) u} dx\right) \right] -\lambda \log (\calP([1-\delta,1])) \\
&= \frac{1}{(1-\delta)^2} I_{\lambda (1-\delta)^2} \left ((1-\delta) u\right) -\lambda \log (\calP([1-\delta,1])). 
\end{split}
\end{equation*}
Hence, for $\lambda (1-\delta)^2>8\pi$, the right hand side of the last inequality is unbounded from below  
in view of \eqref{MT-ifunctunbdd}, and therefore
\[
\inf_{u\in H^1_0 (\Om)}J_\lambda (u) =-\infty  \qquad \mbox{ for any }\lambda >\frac{8\pi}{(1-\delta)^2}.
\]
Since $\delta\in(0,1)$ is arbitrary, \eqref{opt} follows.
\end{proof}
In order to prove Lemma~\ref{mtmigliorata},
we adapt  some ideas contained in \cite{cl}, Proposition~1.   
\begin{proof}[Proof of Lemma~\ref{mtmigliorata}]
Let $g_1,\ldots,g_{\ell+1}$ be smooth functions defined on $\Om$
such that $0\leqslant g_i\le1$, $g_i \equiv 1$ on $\Om_i$, $|\nabla g_i |\le c(d_0)$, for $i=1,\ldots,\ell+1$, 
and $\supp (g_i)\cap\supp (g_j) =\emptyset$ if $i\neq j$. 
Up to relabelling, we may assume that 
\begin{equation}
\label{cancelletto}
\|g_1 \nabla u\|_{L^2(\Om)}\le\| g_i\nabla  u\|_{L^2(\Om)}\qquad\mbox{for any $i=2,\ldots,\ell+1$}.
\end{equation}  
For every $t\in\rr$ we denote $t^+= \max \{0,t\}$. We fix $a>0$.
In view of \eqref{ineq:MT}, applied to $g_1(| u|-a)^+$, we have that
\begin{equation}
\label{passo}
\int_{ \Om}e^{g_1(| u|-a)^+}\le C_{MT}\exp\left\{\frac1{16\pi}\|\nabla \left[ g_1(| u|-a)^+\right] \|^2_{L^2(\Om)}\right \}.
\end{equation}
Hence, in view of assumption~\eqref{omegai} and  of \eqref{passo}, 
using the elementary  inequality $(A+B)^2\le(1+\tau)A^2+c(\tau) B^2$,
for any $\tau>0$, we get
\begin{align*}
\label{mtmigl}
\iint_{I\times\Om}&e^{\alpha u} \Pda dx\le\frac{e^a}{a_0} \iint_{I\times \Om_1}e^{(\alpha u-a)^+} \mathcal{P}(d\alpha )dx\\
&\le\frac{e^a}{a_0} \iint_{I\times \Om}e^{g_1(\alpha u-a)^+} \mathcal{P}(d\alpha) dx\leqslant \frac{e^a}{a_0} \int_{\Om}e^{g_1(| u|-a)^+}dx \\
&\le\frac{ C }{a_0 }  \exp \left\{  \frac1{16 \pi}  \|\nabla\left[  g_1(|u|-a)^+ \right]\|^2_{L^2(\Om)}+a\right\}\\
& \le\frac{ C }{a_0 } \exp \left\{  \frac1{16 \pi} \left[ (1+\tau) \|g_1\nabla u \|^2_{L^2(\Om)} +c(\tau) \| (| u|-a)^+\nabla g_1\| _{L^2(\Om)}^2\right]
+a\right\}\\
&\overset{\eqref{cancelletto}}{\le}\frac{ C }{a_0 }  \exp \left\{  \frac1{16 \pi(\ell+1)} \left[ (1+\tau) 
\sum_{i=1}^{\ell+1}\|g_i\nabla u \|^2_{L^2(\Om)}+c(\tau,d_0) \| (| u|-a)^+\| _{L^2(\Om)}^2\right]+a\right\}\\
&\le\frac{ C }{a_0} \exp \left\{  \frac1{16 \pi(\ell+1)} 
\left[ (1+\tau) \|\nabla u \|^2_{L^2(\Om)}+c(\tau,d_0) \| (| u|-a)^+\| _{L^2(\Om)}^2\right]+a\right\},
\end{align*}
where we used $\supp (g_i)\cap\supp (g_j)=\emptyset$ to derive the last inequality and
where $C=C_{MT}$.
\par
For a given real number $\eta\in(0,|\Om|)$, let $a$ be such that $\meas(\{x\in \Om : | u(x)|\ge a\}) =\eta$. 
Then, by the H\"{o}lder and Sobolev inequalities we have
\begin{equation*}
\begin{split}
 \|(| u|-a)^+\|^2_{L^2(\Om)}=\eta^{1/2}\left(\int_{\{x\in \Om:|u|>a\}}(|u|-a)^4\right)^{1/2}\le\eta^{1/2}C\|\nabla u\|^2_2.
\end{split}
\end{equation*}
Using the Schwarz and Poincar\'e inequalities, we finally derive
\begin{equation*}
a\eta\le\int_{\{x\in\Om:|u|\ge a\}}|u|\le|\Om|^{1/2}\left( \int_\Om |u|^2\right)^{1/2}\le C\|\nabla u\|_2,
\end{equation*}
and therefore, for any small $\delta>0$,
\begin{equation*}
a \le\frac \delta 2\|\nabla u \|_2^2 +\frac{C^2}{2\delta \eta^2}.
\end{equation*}
In conclusion, we have derived that 
\[
\iint_{I\times\Om}e^{\alpha u} \Pda dx\le\frac{ C }{a_0}\exp\left\{\frac1{16 \pi(\ell+1)} 
\left(1+\tau+c(\tau,d_0)\eta^{\frac12}C+\frac\delta2\right)\|\nabla u \|^2_{L^2(\Om)}+\frac{C^2}{2\delta \eta^2} \right\}.
\]
Let $\tilde\ee=\frac{\ee}{16\pi(\ell+1)-\ee}$. Fixing $\tau<\frac{\tilde\varepsilon}3$, 
$\eta$ such that $c(\tau,d_0)\eta^{\frac12}C<\frac{\tilde\ee}{3}$ and $\delta$ such that $\frac{\delta}{2}<\frac{\tilde\ee}{3}$, 
the asserted improved Moser-Trudinger inequality \eqref{tesimt} is completely established.
\end{proof}

%
Using Lemma~\ref{mtmigliorata} we can characterize the limiting behavior of sequences of measures on $\Omega$ of the form 
\begin{equation}
\label{def:fconc}
\frac{\int_Ie^{\alpha u_{n}}\mathcal{P}(d\alpha)dx}{\iint_{I\times \Omega}e^{\alpha u_{n}}\mathcal{P}(d\alpha)dx}, 
\end{equation}
where the functions $u_n$ are such that the functional $J_\la$ attains arbitrarily large negative values.
Such a characterization will be used in an essential way in the variational scheme, in particular
in the proof of Proposition \ref{prop:clambdaboundedfrombelow} below.
\begin{lemma}[Concentration property]
\label{lemma:beta_i}
Assume that $\calP$ satisfies \eqref{assumpt:suppP}.
Let $\lambda\in(8k\pi,8(k+1)\pi)$, $k\geq1$, and let $\{u_n\}\subset H^1_0(\Omega)$ be a sequence of functions satisfying $J_\lambda(u_n)\to-\infty$. For any $\ee>0$ and for any $r>0$, there exists a subsequence $\{u_{n_j}\}$ (depending only on $\ee$ and $r$) and $\ell$ points, $\ell\in\{1,\ldots,k\}$, $p_1,\ldots,p_\ell\in\overline\Omega$ (which do not depend on $j$) such that
\begin{equation}\label{condizionelemmabeta_i}
\frac{\iint_{I\times (\Omega\setminus\cup_{i=1}^\ell B_r(p_i))}e^{\alpha u_{n_j}}\mathcal{P}(d\alpha)dx}{\iint_{I\times \Omega}e^{\alpha u_{n_j}}\mathcal{P}(d\alpha)dx}<\ee\qquad\mbox{for any $i\in\{1,\ldots,\ell\}$ and for any $j$}
\end{equation}
and
\begin{equation}\label{limitbeta_i}
\lim_{j\to+\infty}\left( \frac{\iint_{I\times (B_r(p_i)\setminus\cup_{h=1}^{i-1} B_r(p_h))}e^{\alpha u_{n_j}}\mathcal{P}(d\alpha)dx}{\iint_{I\times \Omega}e^{\alpha u_{n_j}}\mathcal{P}(d\alpha)dx} \right)=\beta_i>0\qquad\mbox{for any $i\in\{1,\ldots,\ell\}$,}
\end{equation}
with $\sum_{i=1}^\ell \beta_i=1$.
\end{lemma}
\begin{proof}
The proof is a direct consequence of some general concentration properties of $L^1$-functions,
as obtained in Lemma 3.3 in \cite{Malchiodi} and
Lemma 2.4 in \cite{BarDem}, applied to the functions \eqref{def:fconc}.
\end{proof}
\subsection{Properties of some Vandermonde-type maps}
Finally, we collect some results from \cite{BarDem} concerning Vandermonde-type maps. 
Such properties will be needed in order to perform the variational scheme in the proof of Theorem \ref{thm:main},
and in particular to prove Proposition \ref{prop:clambdaboundedfrombelow} below.
\par
Henceforth, for $k\in\N$, we denote by $\underline{z}_k$ the vector $\underline{z}_k=(z_1,\ldots,z_k)\in\C^k$. 
In particular we denote by $\underline{0}_k\in\C^k$ the null vector, i.e., the vector whose entries are all equal to $0\in\C$.
\par
Let $D_k$ be the open unit ball of $\C^k$, namely
\[
D_k=\{\underline{z}_k\in\C^k\,|\, |z_1|^2+\ldots+|z_k|^2<1\}.
\]
Let $\tP_{k}: \C^k\mapsto \C^k$ be the continuous map defined as 
\begin{equation}\label{Psitilde}
\tP_k(\underline{z}_k):=\left(
\begin{array}{ccccccc}
z_1\abs{z_1}&+&z_2\abs{z_2}&+&\cdots&+& z_{  k}\abs{z_{  k}}\\
z_1^2&+&z^2_2&+&\cdots&+&z^2_{  k}\\
& {\vdots}& & {\vdots}& & {\vdots}&\\
z_1^2(\frac{z_1}{\abs{z_1}})^{j-2}&+&z_2^2(\frac{z_2}{\abs{z_2}})^{j-2}&+&\cdots&+&z_{  k}^2(\frac{z_{  k}}{\abs{z_{  k}}})^{j-2}\\
& {\vdots}& & {\vdots}& & {\vdots}&\\
z_1^2(\frac{z_1}{\abs{z_1}})^{k-2}&+&z_2^2(\frac{z_2}{\abs{z_2}})^{k-2}&+&\cdots&+&z_{  k}^2(\frac{z_{  k}}{\abs{z_{  k}}})^{k-2}
\end{array}
\right).
\end{equation}
In \cite{BarDem}, Lemma 4.1, the degree of $\tP_k$ was considered and the following
was established.
\begin{lemma}\label{lemma:degree}
If $k\in \N$, then
$$
\deg\left(\tP_{k},\partial D_k,\underline{0}_k\right)\neq0.
$$
\end{lemma}
Next we recall another useful result obtained in \cite{BarDem}, Lemma 3.3.
\begin{lemma}
\label{lemma:continuità}
Let $\ell\in\N$ and $\underline{\beta}_\ell\in\R^\ell$ such that $\beta_i>0$  for any $i\in\{1,\ldots,\ell\}$.
Suppose that $\underline{z}_\ell\in\C^\ell$ is a solution to
\[
\left\{
\begin{array}{l}
\beta_1z_1+\beta_2 z_2+\ldots+\beta_\ell z_\ell =y_1\\
\beta_1z_1^2+\beta_2 z_2^2+\ldots+\beta_\ell z_\ell^2=y_2\\
\qquad\qquad\ldots\ldots\\
\beta_1z_1^\ell+\beta_2 z_2^\ell+\ldots+\beta_\ell z_\ell^\ell=y_\ell
\end{array}
\right.
\]
where $\underline{y}_\ell\in\C^\ell$. Then $\underline{z}_\ell \to \underline{0}_\ell$ as
$\underline{y}_\ell\to \underline{0}_\ell$.
\end{lemma}
\section{Blow-up analysis and proof of Theorem~\ref{thm:mq}}
\label{sec:blow-up}

Let $\un$ be a concentrating sequence of solutions to $(*)_{\lambda_n}$.
In order to prove Theorem~\ref{thm:mq},
we define
the sequence of measures $\mu_n\in\mathcal M([0,1]\times\Omega)$ on the product space $[0,1]\times\Omega$:
\[
\mu_n(d\al dx):=\la_n\,\frac{e^{\al u_n(x)}}{\iint_{[0,1]\times\Omega}e^{\al u_n(x)}\,\mathcal P(d\al)dx}\,\Pda dx.
\]
Clearly, $\mu_n([0,1]\times\Omega)=\la_n$, therefore there exists a measure
$\mu\in\mathcal M([0,1]\times\Omega)$ such that, up to subsequences, $\mu_n\stackrel{\ast}{\rightharpoonup}\mu$
weakly in $\mathcal M([0,1]\times\Omega)$.
In view of the Brezis-Merle theory \cite{BM}, as adapted in Lemma~\ref{lem:BM}, 
there exists a finite set
$\mathcal S=\{p_1,p_2,\ldots,p_m\}\subset\Omega$ such that the singular part of $\mu$
is supported on $[0,1]\times\mathcal S$.
It follows that there exist $\zeta_i\in\mathcal M([0,1])$, $i=1,\ldots,m$,
and $r\in L^1([0,1]\times\Omega)$ such that the limit measure $\mu$ is of the form:
\begin{equation}
\label{eq:mu}
\mu(d\al dx)=\sum_{i=1}^m\zeta_i(d\al)\delta_{p_i}(dx)+r(\al,x)\,\Pda dx.
\end{equation}
With this notation, the main ingredient in the proof of Theorem~\ref{thm:mq} is the following result.
\begin{prop}
\label{prop:quantization}
Assume \eqref{assumpt:suppP}.
Let $\un$ be a sequence of solutions to $(*)_{\lambda_n}$
with $\la_n\to\la_0$ and suppose that $\mathcal S\neq\emptyset$.
Then, 
\begin{enumerate}
\item[(i)]
$\zeta_i(d\al)=8\pi\de_1(d\al)$,
for all $i=1,2,\ldots,m$.
\item[(ii)]
$r\equiv0$. 
\end{enumerate}
\end{prop}
In order to establish Proposition~\ref{prop:quantization}, we first show that the measures $\zeta_i(d\alpha)$ are concentrated at $1$ (see Lemma \ref{lem:fa} below), namely that $\zeta_i(d\alpha)=n_i\delta_1(d\alpha)$ for some $n_i>0$. Next we provide a quadratic identity for the blow up measures $\zeta_i$ (see Lemma \ref{lem:identity} below), which will involve that $n_i=8\pi$ for any $i=1,\ldots,m$. In turn, even without the additional assumption $\mathcal{P}(\{1\})>0$, the argument outlined in Remark \ref{rem:s=0}, to show that $s\equiv0$ (where $s$ is defined in \eqref{eq:nuconv}), will allow us to conclude that $r\equiv0$.
\par
Let us state and prove two preliminary lemmas.\\
Fix $\eps>0$. For every $\alpha\in[0,1-\eps)$ we define for $x\in\Omega$
\[
f_\al(x):=\frac{\lambda e^{\alpha u}}{\iint_{I\times\Omega}e^{\alpha'u}\,\calP(d\alpha') dx}.
\]
\begin{lemma}
\label{lem:fa}
Assume \eqref{assumpt:suppP}.
For every $\alpha\in[0,1-\eps)$ the following estimate holds:
\[
\int_\Omega f_\al^{(1-\eps)/\al}\,dx\le\left(\frac{\lambda}{\calP([1-\eps,1])|\Omega|}\right)^{(1-\eps)/\al}|\Omega|.
\]
\end{lemma}
\begin{proof}
By definition,
\begin{align}
\label{eq:fa}
\int_\Omega f_\al^{(1-\eps)/\al}\,dx=\Big(\frac{\la}{\iint_{I\times\Omega}e^{\alpha'u}\,\calP(d\alpha') dx}\Big)^{(1-\eps)/\al}
\int_\Omega e^{(1-\eps)u}\,dx.
\end{align}
We observe that, since $u\ge0$ by the maximum principle,
we have $1\le e^{(1-\eps)u}\le e^{\al'u}$, and therefore
\begin{equation}
\label{est:fafirst}
\int_\Omega e^{(1-\eps)u}\,dx\le\frac{1}{\mathcal P([1-\eps,1])}
\iint_{[1-\eps,1]\times\Omega}e^{\al'u}\mathcal P(d\al')dx.
\end{equation}
We also obtain that
\begin{align*}
\iint_{[1-\eps,1]\times\Omega}e^{\al' u}\,\mathcal P(d\al')dx\ge\mathcal P([1-\eps,1])|\Omega|
\end{align*}
and consequently, recalling that $(1-\eps)/\al>1$,
\begin{equation*}
\begin{aligned}
&\left(\frac{1}{\mathcal P([1-\eps,1])|\Omega|}\iint_{[1-\eps,1]\times\Omega}e^{\al' u}\,\mathcal P(d\al')dx\right)^{(1-\eps)/\al}\\
&\qquad\qquad\qquad\qquad\qquad\qquad
\ge\;\frac{1}{\mathcal P([1-\eps,1])|\Omega|}\iint_{[1-\eps,1]\times\Omega}e^{\al' u}\,\mathcal P(d\al')dx.
\end{aligned}
\end{equation*}
In turn, we derive that
\begin{equation}
\label{est:fasecond}
\begin{aligned}
\left(\frac{1}{\iint_{[0,1]\times\Omega}e^{\al' u}\,\mathcal P(d\al')dx}\right)^{(1-\eps)/\al}
\le&\;\left(\frac{1}{\iint_{[1-\eps,1]\times\Omega}e^{\al' u}\,\mathcal P(d\al')dx}\right)^{(1-\eps)/\al}\\
\le&\;\frac{\left(\mathcal P([1-\eps,1])|\Omega|\right)^{1-\frac{1-\eps}{\al}}}{\iint_{[1-\eps,1]\times\Omega}e^{\al' u}\,\mathcal P(d\al')dx}.
\end{aligned}
\end{equation}
Inserting \eqref{est:fafirst} and \eqref{est:fasecond} into \eqref{eq:fa}, we obtain
\begin{align*}
\int_\Omega f_\al^{(1-\eps)/\al}\,dx
\le&\;\lambda^{(1-\eps)/\al}
\frac{(\mathcal P([1-\eps,1])|\Omega|)^{1-\frac{1-\eps}{\al}}}{\iint_{[1-\eps,1]\times\Omega}e^{\al'u}\,\mathcal P(d\al')dx}
\cdot\frac{1}{\mathcal P([1-\eps,1])}\iint_{[1-\eps,1]\times\Omega}e^{\al'u}\,\mathcal P(d\al')dx
\\
=&\left(\frac{\lambda}{\mathcal P([1-\eps,1])|\Omega|}\right)^{(1-\eps)/\al}|\Omega|,
\end{align*}
as asserted.
\end{proof}
\begin{lemma}
\label{lem:identity}
For every $i=1,\ldots,m$ the following identity holds:
\begin{equation}
\label{eq:identity}
8\pi\int_{[0,1]}\zeta_i(d\al)=\Big[\int_{[0,1]}\al\,\zeta_i(d\al)\Big]^2.
\end{equation}
\end{lemma}
\begin{proof}
See \cite{RZDIE}, Theorem~2.2, where \eqref{eq:identity} is derived in a more general context
by using a symmetry argument introduced in \cite{SenbaSuzuki}.
Alternatively, \eqref{eq:identity} may be derived from the classical Pohozaev identity,
see, e.g., \cite{Lin}.
\end{proof}
Now we can prove Proposition~\ref{prop:quantization}.
\begin{proof}[Proof of Proposition~\ref{prop:quantization}]
Proof of (i).
In view of Lemma~\ref{lem:fa}, for any $\eps>0$ we have
\begin{align*}
\iint_{[0,1-\eps]\times\Omega}&\left(\frac{\lambda e^{\alpha u_n}}{\iint_{I\times\Omega} e^{\alpha'u_n}\calP(d\alpha') dx}\right)^{(1-\eps/2)/(1-\eps)}\,\Pda dx\\
&=\int_{[0,1-\eps]}\Pda\int_\Omega f_\al^{(1-\eps/2)/(1-\eps)}\,dx\\
&\le\int_{[0,1-\eps]}\Big(\int_\Omega f_\al^{(1-\eps/2)/\al}\,dx\Big)^{\al/(1-\eps)}|\Omega|^{1-\al/(1-\eps)}\,\Pda\\
&\le
\left\{\left(\frac{\la}{\calP([1-\eps,1])|\Omega|}\right)^{(1-\eps/2)/\al}|\Omega|\right\}^{\al/(1-\eps)}|\Omega|^{1-\al/(1-\eps)}\\
&\le
\left(\frac{\la}{\calP([1-\eps,1])|\Omega|}\right)^{(1-\eps/2)/(1-\eps)}|\Omega|.
\end{align*}
It follows that the sequence of functions
\[
\mu_n(\al,x)=\frac{\lambda e^{\alpha u_n}}{\iint_{I\times\Omega}
e^{\alpha'u_n}\calP(d\alpha') dx}
\]
is uniformly bounded in $L^{(1-\eps/2)/(1-\eps)}([0,1-\eps)\times\Omega)$.
Therefore, for all $i=1,\ldots,m$ we have $\left.\mu(d\al dx)\right\vert_{[0,1-\eps)\times\Omega}\equiv0$
and $\supp(\zeta_i)\subset[1-\eps,1]$
for any $\eps>0$. 
This implies that
$\zeta_i(d\al)=n_i\delta_1(d\al)$ for some $n_i>0$.
In turn, from Lemma \ref{lem:identity} we find 
$8\pi n_i=(n_i^2)$ and therefore
$n_i=8\pi$.
\par
Proof of (ii).
We have, for any sufficiently small $\eps>0$:
\begin{align*}
\iint_{\I\times\Omega}e^{\al\un}\,\Pda dx\ge\mathcal P([1-\eps,1])\int_{\Omega}e^{(1-\eps)\un}\,dx.
\end{align*}
On the other hand, up to subsequences, $u_n\to u_0$ in $W_0^{1,q}(\Omega)$ for any $q\in[1,2)$, 
where in view of \eqref{stimau0} and Part~(i) there holds
\[
u_0(x)\ge4\sum_{i=1}^m\left(\ln\frac{1}{|x-p_i|}+H(x,p_i)\right),
\qquad\hbox{in\ }\Omega\setminus\mathcal S.
\]
In particular,
\[
e^{(1-\eps)u_0(x)}\ge\prod_{i=1}^m\frac{c_0}{|x-p_i|^{4(1-\eps)}},
\]
and therefore $\int_\Omega e^{(1-\eps)u_0}=+\infty$.
Hence, by Fatou's lemma we conclude that 
\begin{equation}
\label{eq:integralblowup}
\iint_{\I\times\Omega}e^{\al u_n}\,\Pda dx\to+\infty
\end{equation}
along a blow-up sequence.
Since by Lemma \ref{lem:BM} $\un\in L^\infty_{\mathrm{loc}}(\Omega\setminus \mathcal S))$, \eqref{eq:integralblowup} 
implies $r\equiv0$ in $[0,1]\times\Omega$.
\end{proof}
\begin{proof}[Proof of Theorem~\ref{thm:mq}]
In view of Proposition~\ref{prop:quantization}, 
the limit \eqref{eq:Thm1} holds true.
We are only left to check 
the quantization property $\la_0\in8\pi\mathbb N$.
This fact readily follows from \eqref{eq:Thm1}.
\end{proof}
\section{The min-max scheme and the proof of Theorem~\ref{thm:main}}
\label{sec:var}
Finally, in this section we complete the proof of Theorem~\ref{thm:main}.
In view of assumption~\eqref{assumpt:Omega}
there exists a simple, closed, smooth, non-contractible 
curve $\Gamma\subset\Omega$.
Moreover, by the Jordan-Schoenflies Theorem~\cite{Schoenflies}, there exists a diffeomorphism 
\[
\chi:\R^2\to\C
\]
such that $\chi(\Gamma)=\partial B_1(0)$ and such that the bounded component of $\R^2\setminus\Gamma$ is mapped in $B_1(0)$. 
Clearly, there exists a point $z_0\in B_1(0)\subset\C$ and a radius $\rho>0$ such that $B_{2\rho}(z_0)\cap\chi(\Omega)=\emptyset$. 
Without loss of generality we may assume that $z_0=0$, so that in conclusion we have
\begin{equation}
\label{ipotesiOmega}
\partial B_1(0)=\chi(\Gamma)\subset\chi(\Omega),\qquad\mbox{while}\qquad B_{2\rho}(0)\cap\chi(\Omega)=\emptyset.
\end{equation}
Via $\chi$, we may also define a simple, regular parametrization of $\Gamma$: 
\begin{equation*}
\gamma:[0,2\pi)\to\Gamma, \qquad \gamma(\theta)=\chi^{-1}(e^{i\theta}).
\end{equation*}
For $u\in H^1_0(\Omega)$ and $j\in\N$, let $m_j:H^1_0(\Omega)\to\C$ be defined by
\begin{equation*}
\label{m_j}
m_j(u)=\frac{\iint_{I\times\Omega}(\chi(x))^je^{\alpha u(x)}\,\mathcal{P}(d\alpha)dx}{\iint_{I\times\Omega} 
e^{\alpha' u(x)}\,\mathcal{P}(d\alpha')dx}=\int_\Omega (\chi(x))^j d\mu(u),
\end{equation*}
where
\[
d\mu(u)=\frac{\int_{I} e^{\alpha u(x)}\,\mathcal{P}(d\alpha)}{\iint_{I\times\Omega} e^{\alpha' u(x)}\,\mathcal{P}(d\alpha')dx}.
\]
For $k\in\N$, let $m: H^1_0(\Omega)\to \C^k$ be the vectorial map 
\[
m(u)=(m_1(u),m_2(u),\ldots,m_k(u)).
\]
We now define, for $k\in\N$, the class of functions which will be used in the min-max argument:
\begin{equation}
\label{Flambda}
\mathcal{F}_\lambda
=\left\{
h\in  \mathcal{C}(D_k,H_0^1(\Omega))\,|\,\begin{array}{ll}
(i)\;\;J_\lambda(h(\underline{z}_k))\to-\infty\;\;\;\mbox{as $\underline{z}_k\to\partial D_k$},\\
(ii)\:\:m\circ h\quad\mbox{can be extended continuously to $\overline{D_k}$}\\ 
(iii)\:m\circ h:\partial D_k\to \C^k \mbox{\quad has non zero degree}
\end{array}
\right\}.
\end{equation}
\begin{prop}
\label{prop:F_lambdanonvuoto}
Assume \eqref{assumpt:suppP}
and let $k\in\N$. For any $\lambda\in(8k\pi,8(k+1)\pi)$ the set $\mathcal{F}_\lambda$ is non-empty.
\end{prop}
In order to prove Proposition~\ref{prop:F_lambdanonvuoto}, we define a suitable test function.
Let $\varepsilon_0>0$ such that $B_{\varepsilon_0}(\gamma(\theta))\subset\Omega$ for any $\theta\in[0,2\pi)$ and for $(r,\theta)\in[0,1)\times[0,2\pi)$
let 
\begin{equation*}
\label{def:v}
v_{r,\theta}(x)=\left\{
\begin{array}{ll}
0 & \mbox{if $x\in\Omega\setminus B_{\varepsilon_0}(\gamma(\theta))$}\\
4\log(\frac{\varepsilon_0}{|x-\gamma(\theta)|}) & \mbox{if $x\in B_{\varepsilon_0}(\gamma(\theta))\setminus B_{\varepsilon_0(1-r)}(\gamma(\theta))$}\\
4\log(\frac{1}{1-r}) & \mbox{if $x\in B_{\varepsilon_0(1-r)}(\gamma(\theta))$.}
\end{array}
\right.
\end{equation*}
Let us consider, for $k\in\N$, the following family of probability measures, known in the literature as the set of formal barycenters of $\Gamma$ of order $k$:
\[
\Gamma_k:=\{\sum_{i=1}^k t_i\delta_{\gamma(\theta_i)}\,:\,t_i\in[0,1],\;\sum_{i=1}^k t_i=1,\; \theta_i\in[0,2\pi)\}.
\]
Let us fix $\tilde\al\in I$ satisfying
\begin{equation}
\label{cond:tildeal}
\tilde\alpha\in(\frac34,1)\qquad \hbox{and}\qquad
(2\tilde\alpha-1)\lambda>8k\pi.
\end{equation}
Then, given $\sigma\in \Gamma_k$, $\sigma=\sum_{i=1}^k t_i\delta_{\gamma(\theta_i)}$ and $r\in[0,1)$, we define the function $u_{r,\sigma}\in H^1_0(\Omega)$ by
\begin{equation}\label{ursigma}
u_{r,\sigma}(x)=\frac{1}{\tilde\alpha}\log\left(\sum_{i=1}^{k}t_i e^{\tilde\alpha v_{r,\theta_i}}\right).
\end{equation}
It is readily checked that $u_{r,\sigma}$ depends continuously on $r\in[0,1)$ and $\sigma\in\Gamma_k$.\\
%
%
\begin{lemma}
\label{lemma:exappendix}
Assume \eqref{assumpt:suppP}.
For $\lambda\in(8k\pi,8(k+1)\pi)$, $r\in[0,1)$ and $\sigma\in\Gamma_k$, then
\begin{equation}
\label{app1}
J_\lambda(u_{r,\sigma})\longrightarrow-\infty\mbox{\qquad as $r\to1$, \ \ uniformly for $\sigma\in\Gamma_k$.}
\end{equation}
and
\begin{equation}
\label{app2}
d\mu(u_{r,\sigma})=\frac{\int_{I}e^{\alpha u_{r,\sigma}}\mathcal{P}(d\alpha)}{\iint_{I\times\Omega}e^{\alpha u_{r,\sigma}}\mathcal{P}(d\alpha)dx}
\stackrel{\ast}{\rightharpoonup}\sigma\mbox{\qquad as $r\to1$, \ \ uniformly for $\sigma\in\Gamma_k$,}
\end{equation}
where the function $u_{r,\sigma}$ is defined in \eqref{ursigma} with $\tilde\al$ satisfying \eqref{cond:tildeal}.
\end{lemma}
\begin{proof}
Recalling that 
\[
J_\lambda(u_{r,\sigma})=\frac{1}{2}\int_\Omega |\nabla u_{r,\sigma}|^2 dx-\lambda\log\left(\iint_{I\times\Omega} e^{\alpha u_{r,\sigma}}\mathcal{P}(d\alpha)dx\right),
\]
property~\eqref{app1} will follow from the following two estimates:
\begin{equation}
\label{stimegradiente}
\int_\Omega |\nabla u_{r,\sigma}|^2 dx\leq 32k\pi\log\frac{1}{1-r}+O(1),
\end{equation}
\begin{equation}
\label{stimelogaritmo}
\log\left(\iint_{I\times\Omega} e^{\alpha u_{r,\sigma}}\Pda dx\right)
=(4\tilde\alpha-2)\log\frac{1}{1-r}+\log(\mathcal{P}([\tilde{\alpha},1]))+O(1).
\end{equation}
We note that $\mathcal{P}([\tilde{\alpha},1])>0$ in view of assumption~\ref{assumpt:suppP}.
Indeed, from \eqref{stimegradiente}--\eqref{stimelogaritmo} and recalling \eqref{cond:tildeal},
it follows that
\[
J_\lambda(u_{r,\sigma})\leq2(8k\pi-(2\tilde\alpha-1)\lambda)\log\frac{1}{1-r}-\lambda\log(\mathcal{P}([\tilde{\alpha},1]))+O(1)\to-\infty
\qquad\quad\mbox{as $r\to1$.}
\]
Proof of \eqref{stimegradiente}.
By definition of $v_{r,\theta_i},u_{r,\sigma}$ as in \eqref{def:v}--\eqref{ursigma}, we have
\[
\nabla u_{r,\sigma}(x)=\frac{\sum_{i=1}^kt_i e^{\tilde\alpha v_{r,\theta_i}(x)}\nabla v_{r,\theta_i}(x)}{\sum_{i=1}^kt_i e^{\tilde\alpha v_{r,\theta_i}(x)}},
\]
and
\[
|\nabla v_{r,\theta_i}|=\left\{\begin{array}{ll}
\frac{4}{|x-\gamma(\theta_i)|}& x\in B_{\varepsilon_0}(\gamma(\theta_i))\setminus B_{\varepsilon_0(1-r)}(\gamma(\theta_i))\\
0 &\mbox{otherwise.}
\end{array}\right.
\]
Therefore, if $x\in\Omega\setminus\cup_{j=1}^kB_{\varepsilon_0(1-r)}(\gamma(\theta_j))$, then
\begin{equation*}
\label{star}
\abs{\nabla u_{r,\sigma}(x)}\leq\frac{\sum_{i=1}^kt_i e^{\tilde\alpha v_{r,\theta_i}(x)}\frac{4}{|x-\gamma(\theta_i)|}}{\sum_{i=1}^k
t_i e^{\tilde\alpha v_{r,\theta_i}(x)}}\leq\frac{4}{\min_{i=1,2,\ldots,k}|x-\gamma(\theta_i)|}.
\end{equation*}
From the estimate above we also deduce that, for any $x\in\Omega$,
\begin{equation*}
\label{starstar}
\abs{\nabla u_{r,\sigma}(x)}\leq\frac{4}{\varepsilon_0(1-r)}.
\end{equation*}
Then setting $A_i=\{y\in\Omega\,:\,|y-\gamma(\theta_i)|=\min_j|y-\gamma(\theta_j)|\}$, we have
\begin{equation*}
\begin{aligned}
\int_\Omega|\nabla u_{r,\sigma}(x)|^2dx\stackrel{\eqref{star}}{\leq}&\sum_{i=1}^k\int_{B_{\varepsilon_0(1-r)}(\gamma(\theta_i))}|\nabla  u_{r,\sigma}(x)|^2\,dx
+\int_{\Omega\setminus\cup_{j=1}^k B_{\varepsilon_0(1-r)}(\gamma(\theta_j))}\frac{16}{(\min_i|x-\gamma(\theta_i)|)^2}dx\\
\stackrel{\eqref{starstar}}{\leq}&O(1)+\sum_{i=1}^k\int_{A_i\setminus \cup_{j=1}^k B_{\varepsilon_0(1-r)}(\gamma(\theta_j))}\frac{16}{|x-\gamma(\theta_i)|^2}\,dx\\
\leq& O(1)+\sum_{i=1}^k \int_{A_i\setminus B_{\varepsilon_0(1-r)}(\gamma(\theta_i))}\frac{16}{|x-\gamma(\theta_i)|^2}\,dx\\
\leq &O(1)+16k \int_{B_{\diam(\Omega)}(\gamma(\theta_i))\setminus B_{\varepsilon_0(1-r)}(\gamma(\theta_i))}\frac{dx}{|x-\gamma(\theta_i)|^2}\\
\leq &O(1)+32k\pi\log\frac{1}{1-r},
\end{aligned}
\end{equation*}
so that \eqref{stimegradiente} is proved.
\par
Proof of \eqref{stimelogaritmo}.
It is readily checked that
\begin{equation}
\label{stima1}
\begin{aligned}
\log\left(\iint_{I\times\Omega}e^{\alpha u_{r,\sigma}}\mathcal{P}(d\alpha)dx\right)
\geq&\log\left(\iint_{[\tilde\alpha,1]\times\Omega}e^{\tilde\alpha u_{r,\sigma}}\mathcal{P}(d\alpha)dx\right)\\
=&\log\left(\int_\Omega e^{\tilde\alpha u_{r,\sigma}}dx\right)+\log(\mathcal{P}([\tilde\alpha,1])).
\end{aligned}
\end{equation}
Then, recalling that in view of \eqref{cond:tildeal} we have $\tilde\alpha>\frac12$,
\begin{eqnarray}
\label{stima2}
\int_{\Omega} e^{\tilde\alpha u_{r,\sigma}}dx&=&\int_\Omega \sum_{i=1}^k t_i e^{\tilde\alpha v_{r,\theta_i}}dx\nonumber\\
&=&\sum_{i=1}^k t_i\left[\int_{B_{\varepsilon_0(1-r)}(\gamma(\theta_i))}\frac{dx}{(1-r)^{4\tilde\alpha}}+\int_{B_{\varepsilon_0}(\gamma(\theta_i))\setminus B_{\varepsilon_0(1-r)}(\gamma(\theta_i))}\left(\frac{\varepsilon_0}{|x-\gamma(\theta_i)|}\right)^{4\tilde\alpha}dx\right.\nonumber\\
& &\left.\qquad+\int_{\Omega\setminus B_{\varepsilon_0}(\gamma(\theta_i))}dx\right]\nonumber\\
&=&\sum_{i=1}^k t_i \left[\pi\varepsilon_0^2(1-r)^{2-4\tilde\alpha}+2\pi\varepsilon_0^{4\tilde\alpha}\int_{\varepsilon_0(1-r)}^{\varepsilon_0}\frac{d\rho}{\rho^{4\tilde\alpha-1}}+|\Omega|-\pi\varepsilon_0^2\right]\nonumber\\
&=&\frac{C}{(1-r)^{4\tilde\alpha-2}}+O(1).
\end{eqnarray}
Finally, combining \eqref{stima1} and \eqref{stima2} we obtain \eqref{stimelogaritmo}. Hence, \eqref{app1} is completely established.
\par
Proof of \eqref{app2}.
Let $\varepsilon_0>0$ be such that $B_{\varepsilon_0}(\gamma(\theta))\subset\Omega$ for any $\theta\in[0,2\pi)$ 
and let $\sigma=\sum_{i=1}^k t_i\delta_{\gamma(\theta_i)}$. 
Without loss of generality we may assume that there exists $m=m(\sigma)\leq k$ such that $t_i>0$ for any $i=1,\ldots,m$ and $t_i=0$ for $i>m$.
\par
In order to prove \eqref{app2} it suffices to show that for every $\varepsilon\in(0,\varepsilon_0)$
\begin{equation}
\label{fragola}
\lim_{r\to1}\int\limits_{\cup_{i=1}^m B_{\varepsilon}(\gamma(\theta_i))}d\mu(u_{r,\sigma})=1\qquad\quad\mbox{uniformly with respect to $\sigma\in\Gamma_k$.}
\end{equation}
Let us fix $\varepsilon\in(0,\varepsilon_0)$, let $\delta=\delta(\varepsilon)\in(0,1)$ such that $B_\delta(0)\subset\varphi_\theta(B_\varepsilon(\gamma(\theta)))$, where \linebreak$\varphi_\theta(x)=\frac{x-\gamma(\theta)}{\varepsilon_0}$.\\
We write
\begin{equation}
\label{ABC}
\int\limits_{\cup_{i=1}^m B_{\varepsilon}(\gamma(\theta_i))}d\mu(u_{r,\sigma})=\frac{A+B}{A+C}
\end{equation}
where 
\[
A=\iint_{I\times \cup_{i=1}^m \varphi^{-1}_{\theta_i}(B_\delta(0))}e^{\alpha u_{r,\sigma}}\mathcal{P}(d\alpha)dx,
\]
and 
\[
\begin{aligned}
B=&\iint_{I\times (\cup_{i=1}^m (B_{\varepsilon}(\gamma(\theta_i))\setminus\varphi^{-1}_{\theta_i}(B_\delta(0))))}e^{\alpha u_{r,\sigma}}\mathcal{P}(d\alpha)dx,\\
C=&\iint_{I\times ( \Omega\setminus\cup_{i=1}^m\varphi^{-1}_{\theta_i}(B_\delta(0)))}e^{\alpha u_{r,\sigma}}\mathcal{P}(d\alpha)dx.
\end{aligned}
\]
We claim that, as $r\to1$,
\begin{equation}
\label{est:ABC}
A\to+\infty,\qquad B=O(1),\qquad C=O(1).
\end{equation}
In view of assumption~\eqref{assumpt:suppP}, we have $\calP([\tilde\alpha,1])>0$. Then, for any $r>1-\delta$ we have
\begin{eqnarray*}
A&\geq&\int_{\tilde\alpha}^1\int_{\cup_{i=1}^m \varphi^{-1}_{\theta_i}(B_\delta(0))}e^{\alpha u_{r,\sigma}}\mathcal{P}(d\alpha)dx\\
&\geq& \mathcal P([\tilde\alpha,1])\int_{\cup_{i=1}^m \varphi^{-1}_{\theta_i}(B_\delta(0))}e^{\tilde\alpha u_{r,\sigma}}dx\\
&=&\mathcal P([\tilde\alpha,1])\int_{\cup_{i=1}^m \varphi^{-1}_{\theta_i}(B_\delta(0))}\sum_{j=1}^m t_j e^{\tilde\alpha v_{r,\theta_j}}dx\\
&\geq&\mathcal P([\tilde\alpha,1])\sum_{j=1}^m t_j\int_{\varphi^{-1}_{\theta_j}(B_\delta(0))}e^{\tilde\alpha v_{r,\theta_j}}dx\\
&=&\varepsilon_0^2\mathcal P([\tilde\alpha,1])\sum_{j=1}^m t_j\left[\int_{B_{1-r}(0)}\frac{dy}{(1-r)^{4\tilde\alpha}}+\int_{B_\delta(0)\setminus B_{1-r}(0)}\frac{dy}{|y|^{4\tilde\alpha}}\right]\\
&=&\pi\varepsilon^2_0\mathcal P([\tilde\alpha,1])\left[(1-r)^{2-4\tilde\alpha}+\frac{1}{2\tilde\alpha-1}((1-r)^{2-4\tilde\alpha}-\delta^{2-4\tilde\alpha})\right]\to+\infty\qquad\mbox{uniformly for $\sigma\in\Gamma_k$.}
\end{eqnarray*}
In the last line we have used that $\sum_{j=1}^m t_j=1$ and that $\tilde\alpha>\frac12$.\\
Moreover
\begin{eqnarray*}
0\leq B\leq C&\leq&|\Omega|+\int_{\cup_{i=1}^m(B_{\varepsilon_0}(\gamma(\theta_i))\setminus \varphi^{-1}_{\theta_i}(B_\delta(0)))}
(\sum_{j=1}^m t_j e^{\tilde\alpha v_{r,\theta_j}})^{1/\tilde\alpha}\,dx\\
&\leq&|\Omega|+\int_{\cup_{i=1}^m(B_{\varepsilon_0}(\gamma(\theta_i))\setminus \varphi^{-1}_{\theta_i}(B_\delta(0)))}
(\sum_{j=1}^m  \frac{t_j}{\delta^{4\tilde\alpha}})^{1/\tilde\alpha}\,dx\\
&\leq&|\Omega|+k\pi\frac{\varepsilon_0^2}{\delta^4}.
\end{eqnarray*}
Hence, \eqref{est:ABC} is established.
Letting $r\to1$ in \eqref{ABC}, we obtain \eqref{fragola} and, in turn, \eqref{app2}. This concludes the proof.
\end{proof}
\begin{proof}[Proof of Proposition~\ref{prop:F_lambdanonvuoto}]
Let us consider a continuous function $\eta:[0,1]\to[0,1]$ such that $\eta([0,\frac13])=0$ and $\eta([\frac23,1])=1$ 
and let us introduce the map $h:D_k\to H^1_0(\Omega)$ as 
\begin{equation}
\label{h}
h(\underline{z}_k)=\eta(|\underline{z}_k|)\,u_{|\underline{z}_k|^2,\sigma(\underline{z}_k)},
\end{equation}
where $\sigma(\underline{z}_k)=\frac{\sum_{i=1}^k |z_i|^2\delta_{\gamma(\theta_i)}}{|\underline{z}_k|^2}$ and $\underline{z}_k=(z_1,\ldots,z_k)=(|z_1| e^{i\theta_1},\ldots,|z_k| e^{i\theta_k})$.
\par
We claim that $h\in\mathcal F_\la$.
Indeed, by means of \eqref{app1}--\eqref{app2} it is immediate to see that $h$ satisfies property \eqref{Flambda}--\emph{(i)}.
Moreover,
\[
d\mu(h(\underline{z}_k))\to |z_1|^2\delta_{\gamma(\theta_1)}+\ldots+|z_k|^2\delta_{\gamma(\theta_k)}\;\;\;\mbox{as $\underline{z}_k\to\partial D_k$},
\]
which in turn this implies that
\[
m\circ h(\underline{z}_k)\to \tilde \Psi_k(\underline{z}_k)\qquad \mbox{as $\underline{z}_k\to \partial D_k$},
\]
where $\tP$ is defined in \eqref{Psitilde}, so that \eqref{Flambda}--\emph{(ii)} is also fulfilled.
Finally, by Lemma \ref{lemma:degree}, we also deduce property 
\eqref{Flambda}--\emph{(iii)}.
\end{proof}
We are now ready to define, for $\lambda\in(8k\pi,8(k+1)\pi)$, the min-max value:
\begin{equation}
\label{def:clambda}
c_\lambda=\inf_{h\in\mathcal{F}_\lambda}\sup_{u\in h(D_k)} J_\lambda(u).
\end{equation}
In view of Proposition \ref{prop:F_lambdanonvuoto}, we have $c_\lambda<+\infty$. 
The following lower bound relies in an essential way on the non-contractibility of $\Omega$
as assumed in \eqref{assumpt:Omega}.
\begin{prop}\label{prop:clambdaboundedfrombelow}
Assume \eqref{assumpt:suppP}--\eqref{assumpt:Omega}.
Let $k\in\N$ and $\lambda\in(8k\pi,8(k+1)\pi)$, then $c_\lambda>-\infty$.
\end{prop}
\begin{proof}
The case $k=1$ has been treated in \cite{RiZeJDE}, while the case $k>1$ can be worked out following \cite{BarDem} with minor modifications. 
We prove it for reader's convenience.
\par
We assume by contradiction that for any $n\in\N$ there exists $h_n\in\mathcal{F}_\lambda$ such that 
\[
\sup_{u\in h_n(D_k)}J_\lambda(u)\leq-n.
\]
In view of property \eqref{Flambda}--\emph{(iii)} in the definition of $\mathcal{F}$, 
for any $n\in\N$ we can find $u_n\in h_n(D_k)\subset H^1_0(\Omega)$ such that 
\[
J_\lambda(u_n)\leq-n,\qquad\mbox{and}\qquad m(u_n)=\underline{0}_k.
\]

Next we can apply Lemma \ref{lemma:beta_i} with $r$ and $\varepsilon$ to be chosen later in a convenient way. 
Denoting by $\ell$ the positive integer (less or equal than $k$) found in the above mentioned Lemma, for any $j\in\{1,\ldots,\ell\}$
\begin{eqnarray}\label{mjun}
0=m_j(u_n)&=&\int_\Omega (\chi(x))^jd\mu(u_n)=\int_{\cup_{i=1}^\ell B_r(p_i)}(\chi(x))^jd\mu(u_n)
+\int_{\Omega\setminus\cup_{i=1}^\ell B_r(p_i)}(\chi(x))^jd\mu(u_n)\nonumber\\
&=&\sum_{i=1}^\ell\int_{B_r(p_i)\setminus \cup_{h=1}^{i-1}B_r(p_h)}(\chi(x))^jd\mu(u_n)+\int_{\Omega\setminus \cup_{i=1}^\ell B_r(p_i)} (\chi(x))^jd\mu(u_n)\nonumber\\
&=&\sum_{i=1}^\ell \beta_i (\chi(p_i))^j - R_{j,n}(r)
\end{eqnarray}
where $\beta_i$ and $p_i$ are obtained via Lemma \ref{lemma:beta_i} and $R_{j,n}(r)$, up to a subsequence, can be estimated as follows:
\begin{eqnarray*}
&\abs{R_{j,n}(r)}&=\abs{\sum_{i=1}^\ell\left(\int_{B_r(p_i)\setminus
\cup_{h=1}^{i-1}B_r(p_h)}(\chi(x))^j  d\mu(u_n)-(\chi(p_i))^j\beta_i\right)+
\int_{\Omega\setminus\cup_{i=1}^\ell B_r(p_i)}(\chi(x))^j  d\mu(u_n)}\\
&\stackrel{\eqref{limitbeta_i}+\eqref{condizionelemmabeta_i}}{\leq}&\sum_{i=1}^\ell\int_{B_r(p_i)\setminus
\cup_{h=1}^{i-1}B_r(p_h)}\abs{(\chi(x))^j -(\chi(p_i))^j} d\mu(u_n)+o_n(1)+\varepsilon\\
&\leq&\sum_{i=1}^\ell \int_{B_r(p_i)\setminus\cup_{h=1}^{i-1}B_r(p_h)}\abs{\chi(x)-\chi(p_i)}
\left(\sum_{h=0}^{j-1}\abs{\chi(x)}^{j-1-h}\abs{\chi(p_i)}^{h}\right) d\mu(u_n)+o_n(1)+\varepsilon\\
&\leq&\sum_{i=1}^\ell j \, d^{j-1}\int_{B_r(p_i)\setminus\cup_{h=1}^{i-1}B_r(p_h)}\abs{\chi(x)-\chi(p_i)} d\mu(u_n)+o_n(1)+\varepsilon\\
&\leq&\sum_{i=1}^\ell \ell \, d^{j-1}\,C_\chi\,r\,\beta_i+o_n(1)+\varepsilon=\ell \, d^{j-1}\,C_\chi\,r+o_n(1)+\varepsilon.
\end{eqnarray*}
In the above chain of inequalities $d:=\max\limits_{x\in\Omega}|\chi(x)|$
and $C_\chi=\max\limits_{x_1,x_2\in\Omega}\frac{|\chi(x_1)-\chi(x_2)|}{|x_1-x_2|}$.
Denoting by $z_i:=\chi(p_i)\in\C$, for $i\in\{1,\ldots,\ell\}$, we get, by virtue of \eqref{mjun}, that the $z_i$'s satisfy
\begin{equation}\label{sistemafinale}
\left\{
\begin{array}{l}
\beta_1z_1+\beta_2 z_2+\ldots+\beta_\ell z_\ell =R_{1,n}(r)\\
\beta_1z_1^2+\beta_2 z_2^2+\ldots+\beta_\ell z_\ell^2=R_{2,n}(r)\\
\qquad\qquad\ldots\ldots\\
\beta_1z_1^\ell+\beta_2 z_2^\ell+\ldots+\beta_\ell z_\ell^\ell=R_{\ell,n}(r).
\end{array}
\right.
\end{equation}
By our choice of $\chi$, see \eqref{ipotesiOmega}, there exists $\rho>0$ such that  $\chi(\Omega)\cap B_{2\rho}(0)=\emptyset$, then
\begin{equation}
\label{penultima}
2\rho\leq|\chi(p_i)|=|z_i|.
\end{equation}
On the other hand, by applying Lemma \ref{lemma:continuità} to system \eqref{sistemafinale}, we obtain that there exists $\delta>0$ such that if
\begin{equation}\label{ultima}
|R_{j,n}(r)|\leq \delta\qquad\mbox{for some $n\in\N$ and for any $j\in\{1,\ldots,\ell\}$},
\end{equation}
then $|z_i|\leq\rho$, which would be a contradiction against \eqref{penultima}. Finally, it is immediate to see that choosing $r=\frac{\delta}{2\ell d^{j-1}C_\chi}$, $\varepsilon=\frac{\delta}{2}$ and $n$ sufficiently large condition \eqref{ultima} is fulfilled for any $j\in\{1,\ldots,\ell\}$. The proof is thereby complete.
\end{proof}
Finally, we are able to prove the existence result.
\begin{proof}[Proof of Theorem~\ref{thm:main}]
By the definition \eqref{eq:Nerifunctional} of $J_\lambda$,
it is readily checked that if $\la'\le\la$, then $\mathcal F_{\la'}\subset\mathcal F_\la$
and $J_{\la'}(u)\ge J_\la(u)$ for all $u\in H_0^1(\Omega)$.
Consequently, $c_{\la'}\ge c_\la$, where $c_\la$ is the min-max value defined in \eqref{def:clambda}.
In particular, the mapping $\la\to c_\la$ is monotone, and therefore the derivative $c'_\la$
exists for almost every $\la\in(8k\pi,8(k+1)\pi)$.
We fix $\la\in(8k\pi,8(k+1)\pi)$ such that $c'_\la$ is well-defined.
By the well-known Struwe Monotonicity Trick \cite{Struwe}, a bounded Palais-Smale sequence,
whose bounds depend on $|c'_\la|$,
may be constructed at level $c_\la$.
The details of this construction in the context of mean field equations may be found in 
\cite{StTa}, see also \cite{RiZeJDE} for the specific context of \ref{eq:Neribis}.
By compactness of the Moser-Trudinger embedding, we obtain from the bounded Palais-Smale sequence
a solution to \ref{eq:Neribis}.
In this way, we obtain a solution to \ref{eq:Neribis} for almost every $\la\in(8k\pi,8(k+1)\pi)$.
Now we fix $\la_0\in(8k\pi,8(k+1)\pi)$. Let $\la_n\to\la_0$ be such that $(\ast)_{\la_n}$
admits a solution $u_n$ for all $n$. In view of Theorem~\ref{thm:mq}, we conclude that
the sequence $u_n$ is compact, and consequently there exists a solution $u_0$
to $(\ast)_{\la_0}$ such that $u_n\to u_0$.
In particular, we obtain a solution for $(\ast)_{\la_0}$.
We conclude that solutions to \ref{eq:Neribis} exist for all values $\la\in(8k\pi,8(k+1)\pi)$,
as asserted.
\end{proof}
\section{Appendix: Liouville bubble limit profiles}
\label{sec:appendix}
In view of the mass quantization property, as stated in Theorem~\ref{thm:mq},
it is natural to expect that, upon rescaling, a concentrating sequence $\un$ of solutions to $(*)_{\lambda_n}$ should 
yield a \emph{Liouville bubble} profile, 
namely a profile of the form
\begin{equation}
\label{def:bubble}
U_{\de,\xi}(x)=\ln\frac{8\de^2}{(\de^2+|x-\xi|^2)^2},
\qquad \de>0,\ \xi\in\rr^2.
\end{equation}
This is indeed the case, as we show in this Appendix. However, it turns out
that the usual rescaling yields the desired profile only in 
the \lq\lq non-degenerate" case where $\calP(\{1\})>0$. On the other hand, if
$\calP(\{1\})=0$, such a rescaling yields a trivial profile in the limit, and some extra care is needed in order to 
capture the Liouville bubble profile.
\par
More precisely, let $\un$ be a concentrating sequence of solutions to $(*)_{\lambda_n}$.
It is convenient to set
\[
\In=\iint_{[0,1]\times\Omega}e^{\al'\un}\,\mathcal P(d\al')dx.
\]
We recall that by the maximum principle $\In\ge|\Omega|$, and along a concentrating sequence we have
$\In\to+\infty$.
Then, problem~$(*)_{\lambda_n}$ takes the form
\begin{equation*}
\label{eq:n}
\left\{\begin{aligned}
-\Delta\un=&\frac{\lan}{\In}\int_{[0,1]}\al e^{\al\un}\,\mathcal P(d\al)&&\hbox{in}\ \Omega\\
\un=&0&&\hbox{on}\ \partial\Omega.
\end{aligned}\right.
\end{equation*}
We assume that $\lan\to\la_0$ and that
\begin{equation*}
\label{eq:maxu}
\un(\xn)=\max_\Omega\un\to+\infty
\qquad\mbox{as\ }n\to\infty.
\end{equation*} 
In view of \cite{GNN} (see the proof of Lemma~\ref{lem:BM}), we know that $\xn$
stays well-away from $\partial\Omega$.
\subsection{The \lq\lq non-degenerate case" $\boldsymbol{\calP(\{1\})>0}$}
Throughout this subsection  we assume that 
\[
1\in\supp\calP\qquad \hbox{and}\qquad \calP(\{1\})>0.
\]
We define
\[
\wn(x):=\un(x)-\ln\In.
\]
Then, $\wn(\xn)=\max\wn$ and $\wn$ satisfies
\[
-\Delta\wn=\frac{\lan}{\In}\int_{\I}\al e^{\al\un}\,\Pda
=\lan\calP(\{1\})e^{\wn}+\rn,
\]
where
\[
\rn(x):=\frac{\lan}{\In}\int_{[0,1)}\al e^{\al\un}\,\Pda.
\]
Moreover,
\begin{equation*}
\begin{aligned}
\int_\Omega e^{\wn}\,dx=&\int_\Omega e^{\un-\ln\In}\,dx
=\int_\Omega\frac{e^{\un}\,dx}{\iint_{[0,1]\times\Omega}e^{\al'\un}\,\mathcal P(d\al')dx}\\
=&\calP(\{1\})^{-1}\frac{\iint_{\{1\}\times\Omega}e^{\al'\un}\,\mathcal P(d\al')dx}{\iint_{[0,1]\times\Omega}e^{\al'\un}\,\mathcal P(d\al')dx}\\
\le&\calP(\{1\})^{-1}.
\end{aligned}
\end{equation*}
In order to rescale, we set:
\begin{equation*}
\sn:=e^{-\wn(\xn)/2}
\qquad
\tOn:=\frac{\Omega-\xn}{\sn}
\end{equation*}
and
\[
\twn(y):=\wn(\xn+\sn y)+2\ln\sn,\qquad y\in\tOn.
\]
Since $\xn$ stays well-away from the boundary of $\Omega$, the rescaled domain $\tOn$ invades the whole space $\rr^2$.
\par
Then, $\twn$ satisfies
\begin{equation*}
\begin{cases}
-\Delta\twn=\lan\calP(\{1\})e^{\twn}+\lan\trn
&\mbox{in\ }\tOn\\
\twn(y)\le\twn(0)=0\\
\int_{\tOn}e^{\twn}=\int_\Omega e^{\wn}\le\calP(\{1\})^{-1}
\end{cases}
\end{equation*}
where
\[
\trn(y):=\frac{\sn^2}{\In}\int_{[0,1)}\al e^{\al\un(\xn+\sn y)}\,\Pda.
\]
We claim that
\begin{equation}
\label{eq:limtrn}
\|\trn\|_{L^\infty(\tOn)}\to0 \mbox{ as }n\to\infty.
\end{equation}
Indeed, given $\eps>0$, let $\eta>0$ be sufficiently small so that
$\calP([1-\eta,1))<\eps$.
Let $n_0\in\mathbb N$ be sufficiently large so that $(\sn^2/\In)^\eta<\eps$
for all $n\ge n_0$.
We estimate, for all $n\ge n_0$:
\begin{equation*}
\begin{aligned}
\trn(y)=&\int_{[0,1)}\al e^{\al\twn(y)}\left(\frac{\sn^2}{\In}\right)^{1-\al}\,\Pda
\stackrel{\twn(y)\le0}{\le}\int_{[0,1)}\left(\frac{\sn^2}{\In}\right)^{1-\al}\,\Pda\\
=&\int_{[0,1-\eta)}\left(\frac{\sn^2}{\In}\right)^{1-\al}\,\Pda
+\int_{[1-\eta,1)}\left(\frac{\sn^2}{\In}\right)^{1-\al}\,\Pda\\
\le&\left(\frac{\sn^2}{\In}\right)^{\eta}+\calP([1-\eta,1))
<2\eps.
\end{aligned}
\end{equation*}
Hence, \eqref{eq:limtrn} is established.
\par
We conclude that there exists a solution $\tw\in C^2_{\mathrm{loc}}(\rr^2)$ 
to the problem
\begin{equation*}
\begin{cases}
-\Delta\tw=\la_0\calP(\{1\})e^{\tw}
&\mbox{in\ }\rr^2\\
\tw(y)\le\tw(0)=0\\
\int_{\rr^2}e^{\tw}\le\calP(\{1\})^{-1}
\end{cases}
\end{equation*}
such that, up to subsequences,
$\twn\to\tw$ in $C^2_{\mathrm{loc}}(\rr^2)$.
In view of Chen-Li's classification result \cite{ChenLi}, the function $\tw+\log(\la_0\calP(\{1\}))$
is of the form \eqref{def:bubble}.
The asserted limit profile is thus established in the case $\calP(\{1\})>0$.
\subsection{The \lq\lq degenerate case" $\boldsymbol{\calP(\{1\})=0}$}
Throughout this section, we assume
\begin{equation}
\label{assumpt:degP}
1\in\supp\calP\qquad\hbox{and}\qquad\calP(\{1\})=0.
\end{equation}
We show the following.
\begin{prop}
\label{prop:bubble}
Assume \eqref{assumpt:degP}.
There exist a rescaling of $\un$ of the form
\[
\twn(y)=\an\un(\xn+\sn y)-\ln\In,
\qquad y\in\tOn=\frac{\Om-\xn}{\sn}
\]
where, as $n\to\infty$, $\an\to1$, $\sn^2=e^{-\an\un(\xn)+\ln\In}\to0$,
and a solution $\tw$ to the problem
\begin{equation*}
\begin{cases}
-\Delta\tw=\la_0 e^{\tw}&\hbox{in\ }\rr^2\\
\int_{\rr^2} e^{\tw}\,dx<+\infty\\
\tw(y)\le\tw(0)=0,
\end{cases}
\end{equation*}
such that, up to subsequences, $\twn\to\tw$ in $C^2_{\mathrm{loc}}(\rr^2)$. 
In particular, $\tw+\log\la_0$ is of the desired form \eqref{def:bubble}.
\end{prop}
We define $\an\in[0,1]$ and the functions $\wn,\Vn$ by setting:
\begin{equation*}
\label{def:an}
\begin{aligned}
&e^{\an\un(\xn)}:=\int_{[0,1]}\al e^{\al\un(\xn)}\,\calP(d\al)\\
&\wn:=\an\un-\ln\In\\
&\Vn:=\frac{\an\lan}{\In}\int_{[0,1]}\al e^{\al\un}\,\calP(d\al)\,e^{-\wn}.
\end{aligned}
\end{equation*}
With this notation, we have:
\begin{lemma}
\label{lem:an}
Assume \eqref{assumpt:degP}.
The following facts hold true:
\begin{enumerate}
  \item [(i)]
$\an\to1$.
  \item[(ii)] $\Vn(\xn)=\an\lan\to\la_0$
  \item[(iii)] $\|\Vn\|_{L^\infty(\Omega)}\le\an\lan(\int_{[0,1]}\al\,\calP(d\al)+1)$
  \item[(iv)] $\int_\Omega\Vn e^{\wn}\,dx\le\an\lan$. 
\end{enumerate}
\end{lemma}
\begin{proof}
Proof of (i).
There holds:
\begin{equation*}
e^{\an}=\left(\int_{[0,1]}\al e^{\al\un(\xn)}\,\calP(d\al)\right)^{1/\un(\xn)}
=\|\al e^{\al}\|_{L^{\un(\xn)}(I,\calP)}\to\|\al e^{\al}\|_{L^{\infty}(I,\calP)}=e.
\end{equation*}
Hence, $\an\to1$.
\par
Proof of (ii). 
By definition of $\an$ and $\wn$, we have
\begin{equation*}
\Vn(\xn)=\frac{\an\lan}{\In}\int_{[0,1]}\al e^{\al\un(\xn)}\,\Pda\,e^{-\wn(\xn)}
=\frac{\an\lan}{\In}e^{\an\un(\xn)}e^{-\wn(\xn)}=\an\lan.
\end{equation*}
\par
Proof of (iii).
By definition, we have
\begin{equation*}
\Vn=\an\lan\int_{\I}\al e^{(\al-\an)\un}\,\Pda.
\end{equation*}
Since $\un\ge0$, we estimate:
\begin{equation*}
\begin{aligned}
\int_{\I}\al e^{(\al-\an)\un}\,\Pda
=&\int_{\al<\an}\al e^{(\al-\an)\un}\,\Pda+\int_{\al\ge\an}\al e^{(\al-\an)\un}\,\Pda\\
\le&\int_{\al<\an}\al\,\Pda+\int_{\al\ge\an}\al e^{(\al-\an)\un(\xn)}\,\Pda\\
\le&\int_{\I}\al\,\Pda+\int_{\I}\al e^{(\al-\an)\un(\xn)}\,\Pda\\
=&\int_{\I}\al\,\Pda+1.
\end{aligned}
\end{equation*}
The asserted estimate follows.
\par
Proof of (iv). 
By definition of $\In$, $\Vn$, we have
\begin{equation*}
\int_\Omega\Vn e^{\wn}=\frac{\an\lan}{\In}\iint_{I\times\Om}\al e^{\al\un}\,\Pda\le\an\lan.
\end{equation*}
The asserted estimates are established.
\end{proof}
Now, we can prove Proposition~\ref{prop:bubble}.
\begin{proof}[Proof of Proposition~\ref{prop:bubble}]
We define the rescaling:
\[
\sn:=e^{\wn(\xn)/2},
\qquad
\twn(y):=\wn(\xn+\sn y)+2\ln\sn,
\qquad
\tVn(y)=\Vn(\xn+\sn y).
\]
The function $\twn$ satisfies
\begin{equation*}
\begin{cases}
-\Delta\twn=\tVn e^{\twn}&\mbox{in }\tOn\\
\int_{\tOn}\tVn e^{\twn}\le C\\
\twn(y)\le\twn(0)=0\\
\|\tVn\|_{L^\infty(\tOn)}\le C,
\end{cases}
\end{equation*}
where, as above, $\tOn$ invades the whole space $\rr^2$.
In view of the estimates in Lemma~\ref{lem:an},
there exists a solution $\widetilde w$ to the problem
\begin{equation*}
\begin{cases}
-\Delta\tw=\la_0 e^{\tw}\\
\int_{\rr^2}e^{\tw}<+\infty\\
\tw(y)\le\tw(0)=0.
\end{cases}
\end{equation*}
such that a subsequence, still denoted $\twn$, satisfies $\twn\to\widetilde w$
locally uniformly on $\rr^2$.
In view of the classification in \cite{ChenLi}, the function $\tw+\ln\la_0$
is of the form \eqref{def:bubble}.
\par
Hence, Proposition~\ref{prop:bubble} is established.
\end{proof}
In view of Proposition~\ref{prop:bubble}, we expect that concentrating solutions to \ref{eq:Neribis}
may be constructed by the approach in \cite{EGP}.
\section*{Acknowledgements}
This research is partially supported by the following grants:  
PRIN $201274$FYK7$\_005$; Progetto GNAMPA-INDAM 2015 \emph{Alcuni aspetti di equazioni ellittiche non lineari}; 
Sapienza Funds \emph{Avvio alla ricerca 2015}.

\end{document}